\documentclass{amsart}

\usepackage{mathrsfs}
\usepackage{amscd}
\usepackage{amsmath}
\usepackage{amssymb}
\usepackage{amsthm}
\usepackage{epsf}
\usepackage{latexsym}
\usepackage{verbatim}
\usepackage[all, cmtip]{xy}
\usepackage{tikz}
\usetikzlibrary{calc}
\usetikzlibrary{positioning}
\usetikzlibrary{matrix}
\usetikzlibrary {graphs}
\usepackage{float}
\usepackage[hidelinks]{hyperref}
\usepackage{comment}
\usepackage{enumitem}
\usepackage{subcaption}
\usepackage{tikz-cd}
\usepackage[capitalise]{cleveref}

\tikzstyle{bsq}=[rectangle, draw, thick, minimum width=.5cm, minimum height=.5cm]
\tikzstyle{bver}=[rectangle, draw, thick, minimum width=1cm, minimum height=2cm]
\tikzstyle{bhor}=[rectangle, draw, thick, minimum width=2cm, minimum height=1cm]

\tikzstyle{divisor}=[circle,very thick,draw,scale=0.3,fill=white]
\tikzstyle{vertex}=[circle,draw,scale=0.3,fill=black]
\usetikzlibrary{patterns}

\usepackage[left=3.2cm, right=3.2cm]{geometry}

\newtheorem{theorem}{Theorem}[section]
\newtheorem{lemma}[theorem]{Lemma}

\newtheorem{corollary}[theorem]{Corollary}
\newtheorem{proposition}[theorem]{Proposition}

\newtheorem{varexample}[theorem]{Example}

\theoremstyle{definition}

\newtheorem{definition}[theorem]{Definition}

\newcommand{\Z}{\mathbb{Z}}
\newcommand{\R}{\mathbb{R}}

\newcommand{\ord}{\operatorname{ord}}
\newcommand{\Trop}{\operatorname{Trop}}

\newcommand{\ddiv}{\operatorname{div}}

\newcommand{\Jac}{\operatorname{Jac}}

\newcommand{\Pic}{\operatorname{Pic}}

\newcommand{\rk}{\operatorname{rk}}

\newcommand{\WG}{\widetilde{\Gamma}}
\newcommand{\WC}{\widetilde{C}}

\newcommand{\Mod}[1]{\ (\mathrm{mod}\ #1)}

\newenvironment{example}{\begin{varexample}
\begin{normalfont}}{\end{normalfont}
\end{varexample}}

\begin{document}

\title{Lifting Prym Divisors on a Bielliptic Loop of Loops}
\author[Aung]{Phew Aung}
\author[Burkholder]{Eric Burkholder}
\author[Dalton]{Joan Dalton}
\author[Jensen]{David Jensen}
\author[Letamendi]{Jacob Letamendi}
\author[Nguyen]{Brianna Nguyen}
\author[Radicker]{Andrew Radicker}
\author[Rivera Laboy]{Doel Rivera Laboy}

\bibliographystyle{alpha}

\begin{abstract}
Let $C$ be a bielliptic curve over an algebraically closed nonarchimedean field of odd characteristic or characteristic zero, and let $f \colon \WC \to C$ be an \'{e}tale double cover.  We prove that, if $f$ specializes to a certain double cover of a loop of loops, then every Prym divisor on the loop of loops lifts to a Prym divisor of the same rank on $\WC$.  As consequences, we compute the dimension of the Prym-Brill-Noether variety $V^r (C,f)$, show that it is reduced and irreducible when the dimension is positive, and compute its cohomology class.
\end{abstract}

\maketitle

\setcounter{tocdepth}{1}

\section{Introduction}
\label{Sec:Intro}

If $f \colon \WC \to C$ is an \'{e}tale double cover of curves, a \emph{Prym divisor class} on $\WC$ is a divisor class with canonical norm.  The set of Prym divisor classes on $\WC$ has two connected components:
\begin{align*}
P^0 (C,f) &:= \{ [D] \in \Pic (\WC) \mid \mathrm{Nm}_f [D] = K_C, \rk (D) \equiv 0 \Mod{2} \} \\
P^1 (C,f) &:= \{ [D] \in \Pic (\WC) \mid \mathrm{Nm}_f [D] = K_C, \rk (D) \equiv 1 \Mod{2} \} .
\end{align*}
Given a nonnegative integer $r$, the \emph{Prym-Brill-Noether variety} is defined as:
\[
V^r (C,f) := \{ [D] \in P^r (C,f) \mid \rk (D) \geq r \}.
\]
The Prym-Brill-Noether variety has many properties in common with the classical Brill-Noether variety $W^r_d (C)$ parameterizing divisor classes of fixed degree $d$ and rank at least $r$ on a curve $C$.

In \cite{PBN}, the fourth author studies the Prym-Brill-Noether theory of \'{e}tale double covers $f \colon \WC \to C$ where $C$ is a cover of an elliptic curve.  More precisely, a curve $C$ is called $k$-\emph{elliptic} if there exists a degree-$k$ map from $C$ to a curve of genus one that does not factor through an isogeny.  The $k$-elliptic locus in $\mathcal{R}_g$ is the moduli space of \'{e}tale double covers $f \colon \WC \to C$ where $C$ is a $k$-elliptic curve of genus $g$.  For a nonempty open subset of the $k$-elliptic locus in $\mathcal{R}_g$, \cite[Theorem~1.1]{PBN} provides an upper bound on the dimension of the Prym-Brill-Noether variety $V^r (C,f)$.  When $k \geq r$, this upper bound is achieved, meaning that for a nonempty open subset of the $k$-elliptic locus in $\mathcal{R}_g$, \cite[Theorem~1.1]{PBN} computes the dimension of $V^r (C,f)$.  When $k<r$, it remains an open question to compute this dimension.

The approach of \cite{PBN} is via tropical geometry.  The paper starts with a tropical curve $\Gamma$ known as the \emph{loop of loops}, constructs an \'{e}tale double cover $\pi \colon \WG \to \Gamma$, and then completely describes the tropical Prym-Brill-Noether variety $V^r (\Gamma,\pi)$.  If $f \colon \WC \to C$ is an \'{e}tale double cover of curves specializing to $\pi$, then the tropicalization of $V^r (C,f)$ is contained in $V^r (\Gamma,\pi)$.  Since tropicalization preserves dimension, it follows that the dimension of $V^r (C,f)$ is less than or equal to that of $V^r (\Gamma,\pi)$, establishing the bound in \cite[Theorem~1.1]{PBN}.  When $k \geq r$, \cite[Theorem~1.2]{PBN} provides a lifting result, showing that the tropicalization of $V^r (C,f)$ is in fact \emph{equal} to $V^r (\Gamma,\pi)$.  The main result of this paper is an analogous lifting result when $k=2$.

\begin{theorem}
\label{Thm:Lifting}
Let $\Gamma$ be a bielliptic loop of loops of genus $g$ and $\pi \colon \WG \to \Gamma$ the quotient by the antipodal involution.  Let $K$ be an algebraically closed nonarchimedean field of odd residue characteristic or characteristic zero, let $C$ be a bielliptic curve of genus $g$ over $K$, and $f \colon \WC \to C$ an \'{e}tale double cover that specializes to $\pi$.  Then every divisor class in $V^r (\Gamma,\pi)$ lifts to a divisor class in $V^r (C,f)$.
\end{theorem}

Lifting results like Theorem~\ref{Thm:Lifting} are of fundamental importance in tropical geometry.  If one knows that a given tropical object is the tropicalization of a classical algebro-geometric object, then one can hope to use it to prove results in classical algebraic geometry.  Because of this, lifting theorems are ubiquitous in the subject.  For example, lifting theorems have been proven for the Brill-Noether variety of a chain of loops with generic edge lengths \cite{CJP, JensenRanganathan21}, maximal components of the Brill-Noether variety for a $k$-uniform chain of loops \cite{JensenRanganathan21, CookPowellJensen22}, the Prym-Brill-Noether variety of a loop of loops with generic edge lengths \cite{PBN}, harmonic morphisms of tropical curves \cite{ABBR15a}, and abelian covers of tropical curves \cite{LenUlirschZakharov}, just to name a few.  All of these results have been used in applications to classical algebraic geometry \cite{MRC, JensenRanganathan21, CookPowellJensen22b, M23, M13, PBN}.

The proof of Theorem~\ref{Thm:Lifting} involves three main steps.  First, consider a chain of double covers $\WC \xrightarrow{f} C \xrightarrow{h} E$, where $f$ is \'{e}tale, $C$ has genus $g$, and $E$ has genus 1.  The Prym varieties of such covers are studied in \cite{Debarre, Horing, Naranjo}.  In particular, if the cover $h \circ f \colon \WC \to E$ is Galois, then the Galois group is $\Z/2\Z \times \Z/2\Z$, and thus the map $h \circ f$ factors through two additional maps $f_1 \colon \WC \to C_1$ and $f_2 \colon \WC \to C_2$ in addition to $f$.  We let $\mathcal{R}_{B,g,t}$ denote the space of such covers where $g(C_1) = t+1$.  In \cite{Debarre}, Debarre defines a variety $Z_a$ consisting of certain pairs of divisor classes in $\Pic (C_1) \times \Pic (C_2)$ (see Section~\ref{Sec:KnownStuff} for a precise definition).  We define
\[
Y_a := \{ f_1^* [D_1] + f_2^* [D_2] \mid ([D_1], [D_2]) \in Z_a \} .
\]
In Section~\ref{Sec:Bielliptic}, we survey the literature on this subject and prove the following.

\begin{proposition}
\label{Prop:BiellipticGenus1}
Let $[f \colon \WC \to C] \in \mathcal{R}_{B,g,0}$.  If $r \geq 2$, then $V^r (C,f) = Y_{r+1}$.
As a consequence, $V^r (C,f)$ has dimension $g-r-2$ when $g-r-2 \geq 0$ and is empty otherwise.  If $g>r+2$, then $V^r (C,f)$ is irreducible, and if $g=r+2$, then $V^r (C,f)$ consists of two points.
\end{proposition}

In Section~\ref{Sec:LOL}, we turn to tropical geometry.  After briefly surveying the basics of divisor theory on tropical curves, we then discuss the Prym-Brill-Noether theory of the loop of loops, as described in \cite{PBN}.  The main result of this section is that the Prym-Brill-Noether variety of a bielliptic loop of loops has the combinatorial structure of a certain polytope called a \emph{zonotope}.

 \begin{theorem}
\label{Thm:PrymZonotope}
Let $\Gamma$ be a bielliptic loop of loops and $\pi \colon \WG \to \Gamma$ the quotient by the antipodal involution.  If $r \geq 2$, then $V^r (\Gamma,\pi)$ is a union of tori, and the intersection lattice of these tori is the face lattice of a zonotope. 

In the special case where $r=2$, the statement holds for every loop of loops $\Gamma$, regardless of edge lengths.
\end{theorem}

Finally, in Section~\ref{Sec:Lifting}, we prove Theorem~\ref{Thm:Lifting}.  The approach is very similar to the proof of \cite[Theorem~1.2]{PBN}, which itself borrows heavily from \cite{CJP}.  We let $\mathcal{X}$ be the intersection of $V^r (C,f)$ with $g-2-r$ general translates of the theta divisor, and let $X$ be, analogously, the intersection of $V^r (\Gamma,\pi)$ with $g-2-r$ general translates of the tropical theta divisor.  We use the combinatorics described in Section~\ref{Sec:LOL} to prove that $X$ is a finite local complete intersection and compute $\# X$.  We then use Rabinoff's lifting theorem \cite{Rabinoff12} to show that the restriction of the tropicalization map to $\mathcal{X}$ is injective, and by comparing the size of the domain and target, conclude that it must be surjective as well.  While the general outline is the same, we note one minor difference between our proof and those appearing earlier in \cite{CJP, PBN}.  Both of the earlier papers use the known cohomology class of the (Prym)-Brill-Noether variety to compute the number of points in $\mathcal{X}$, but in principle this is unnecessary.  Indeed, it is clear from this proof outline that one only needs $\# \mathcal{X} \geq \# X$.  In this way, rather than using the class of $V^r (C,f)$ to prove the lifting result, we instead use Theorem~\ref{Thm:Lifting} to compute the class.

\begin{corollary}
\label{Cor:Class}
Let $r \geq 2$.  For a general \'{e}tale double cover $[f \colon C] \to \WC \in \mathcal{R}_{B,g,0}$, one has the following equality of cohomology classes:
\[
[ V^r (C,f) ] = 2 \cdot \frac{\Theta^{r+1}}{(r+1)!}.
\]
\end{corollary}

\subsection*{Acknowledgements}

This research was conducted as a project with the University of Kentucky Math Lab, supported by NSF grant DMS–2451533, as well as the Simons Travel Support for Mathematicians program.  We thank Philip Engel for suggesting the degeneration argument in the proof of Lemma~\ref{Lem:Intersection}.

\section{Prym Varieties of Bielliptic Curves}
\label{Sec:Bielliptic}

\subsection{Summary of Known Results}
\label{Sec:KnownStuff}

In this section, we summarize \cite{Debarre} on Prym varieties of bielliptic curves.  Let $C$ be a curve of genus $g$ and let $p \colon C \to E$ be a degree 2 map to a curve of genus 1.  A curve $C$ that admits such a map is called a \emph{biellitpic} curve.  Let $f \colon \WC \to C$ be an \'{e}tale double cover of $C$.  The Galois group of $\WC$ over $E$ is either $\Z/2\Z$ or $\Z/2\Z \times \Z/2\Z$.  In the latter case, we let $\iota$ be the involution that exchanges the two sheets of the map $f$, and let $\iota_1 , \iota_2$ denote the other non-identity elements of the Galois group.  Letting $C_i = \WC / \langle \iota_i \rangle$, we obtain the following commutative diagram:
\begin{equation}
\label{eq:classical}
\xymatrix {
& \WC \ar[ld]_{f_1} \ar[d]^{f} \ar[rd]^{f_2} & \\
C_1 \ar[rd]_{p_1} & C \ar[d]^{p} & C_2 \ar[ld]^{p_2} \\
& E & 
 }
\end{equation}
One has $g(C_1) + g(C_2) = g+1$.  We assume throughout without loss of generality that $g(C_1) \leq g(C_2)$.

We denote by $\mathcal{R}_{B,g}$ the space of \'{e}tale double covers $f \colon \WC \to C$, where $C$ is a bielliptic curve of genus $g$.  Let $\mathcal{R}'_{B,g} \subset \mathcal{R}_{B,g}$ be the space of such covers with Galois group $\Z/2\Z$, and let $\mathcal{R}_{B,g,t} \subset \mathcal{R}_{B,g}$ be the space of such covers with Galois group $\Z/2\Z \times \Z/2\Z$, and where $g(C_1) = t+1$.  The moduli space $\mathcal{R}_{B,g}$ admits a decomposition into irreducible components:
\[
\mathcal{R}_{B,g} = \big( \bigcup_{t=0}^{\lceil \frac{g-1}{2} \rceil } \mathcal{R}_{B,g,t} \big) \cup \mathcal{R}'_{B,g}.
\]

Let $\Delta$ denote the branch divisor of the map $p \colon C \to E$ and let $P_1 , \ldots , P_{2g-2} \in C$ denote the ramification points.  There exists a divisor class $\delta$ on $E$ such that $p^* \delta \sim \sum P_i$.  In particular, this implies that $2\delta \sim \Delta$.  Debarre shows that the Prym-Brill-Noether theory of the cover $f \colon \WC \to C$ is deeply connected to the variety:
\[
Z_a := \{ ([D_1], [D_2]) \in W^0_{g(C_1)-1+a} (C_1) \times W^0_{g(C_2)-1-a} (C_2) \mid \mathrm{Nm}_{p_1} (D_1) + \mathrm{Nm}_{p_2} (D_2) \sim \delta \} .
\]
We let $Y_a$ denote the image of $Z_a$ in $\Pic (\WC)$.  That is,
\[
Y_a := \{ f_1^* [D_1] + f_2^* [D_2] \mid ([D_1], [D_2]) \in Z_a \} .
\]
Debarre proves that $Y_a \subseteq P^r (C,f)$ if and only if $a \equiv r+1 \Mod{2}$.  Moreover, for $r \geq 1$, he proves that
\[
V^r (C,f) \subseteq \bigcup_{\substack{a=1-g(C_1) \\ a \equiv r+1 \Mod{2}}}^{g(C_2)-1} Y_a .
\]

We will use the following key observation.

\begin{lemma}
\label{Lem:Irreducible}
If $g \geq 3$, $a \leq g-2$, and $Z_a$ is nonempty, then $Z_a$ is irreducible of dimension $g-1-a$.  When $g(C_1) = 1$, the variety $Z_{g-1}$ consists of two points.
\end{lemma}

\begin{proof}
This is shown in \cite[Lemma~5.3]{Horing} in all cases except when $g(C_1) = 1$ and $a \geq g-2$.  When $g(C_1) = 1$, the map $p_1 \colon C_1 \to E$ is an isogeny of degree 2, hence $\mathrm{Nm}_{p_1}$ has kernel of order 2.  It follows immediately that
\[
Z_{g-1} = \{ ([D_1], 0) \mid \mathrm{Nm}_{p_1} (D_1) \sim \delta \}
\]
consists of 2 points.

We now show that, if $g(C_1)=1$, then $Z_{g-2}$ is isomorphic to $\WC$.  Since $\WC$ is an irreducible curve, the result will follow.  Let $\delta' \in \mathrm{Nm}_{p_1}^{-1} (\delta)$.  Consider the map $k \colon \WC \to W^0_{g-2} (C_1) \times W^0_1 (C_2)$ given by
\[
k(x) = ( \delta' - \mathrm{Nm}_{f_1} (x), \mathrm{Nm}_{f_2} (x)) .
\]
Note that, since $g \geq 3$, we have $g-2 \geq 1$, and since $g(C_1)=1$, every divisor class of degree $g-2$ on $C_1$ is effective.  Hence, the image of $k$ is contained in $Z_{g-2}$.  Moreover, one has the commutative diagram:
\[
\xymatrix {
\WC \ar[r]^{k} \ar[rd]_{f_1} & Z_{g-2} \ar[d] \\
& C_2 
 }
\]
where the left arrow is an \'{e}tale double cover by construction and the right arrow is an \'{e}tale double cover because $\mathrm{Nm}_{p_1}$ has kernel of order 2.  It follows that $k$ must be an isomorphism.
\end{proof}

We will prove a similar statement about the image $Y_a$ in Corollary~\ref{Cor:Irreducible} below.  Given a map of curves $F \colon Y \to X$, an effective divisor $D$ on $Y$ is $F$-\emph{simple} if there is no point $x \in X$ such that $F^* (x) \leq D$.  A divisor class $[D]$ is $F$-simple if $D'$ is simple for all $D' \in \vert D \vert$.  We will repeatedly use the following fact.

\begin{lemma}\cite[Corollary~5.2.8]{Debarre}
\label{Lem:Simple}
If $D_2$ is $p_2$-simple, then
\[
h^0 (\WC, f_1^* D_1 + f_2^* D_2) \leq 2h^0 (C_1, D_1) + g(C_1) - 1 - \deg (D_1) .
\]
\end{lemma}

We now prove that the variety $Y_a$ is irreducible, except in one case, and of the expected dimension, except in one case.

\begin{corollary}
\label{Cor:Irreducible}
If $1 \leq a \leq g-2$ and $Y_a$ is nonempty, then it is irreducible of dimension $g-1-a$.  When $g(C_1) = 1$, $Y_{g-1}$ consists of two points.
\end{corollary}

\begin{proof}
For $a \leq g-2$, $Z_a$ is irreducible by Lemma~\ref{Lem:Irreducible}, hence its image $Y_a$ is irreducible as well.  When $a=0$, by \cite[Proposition~5.2.1]{Debarre}, we have $\mathrm{dim} (Y_a) = g-2$, hence the condition $a \geq 1$ in the statement above.  For the remaining statements, it suffices to prove that the map from $Z_a$ to $Y_a$ is generically injective.  If $a \geq 1$, then for a general element $([D_1],[D_2]) \in Z_a$, we have $[D_2]$ is $p_2$-simple and $h^0 (C_1,D_1) = a$.  Hence, by Lemma~\ref{Lem:Simple}, we have
\[
h^0 (\WC, f_1^* D_1 + f_2^* D_2) \leq 2h^0 (C_1, D_1) - a = h^0 (C_1,D_1).
\]
In other words, every section of $f_1^* D_1 + f_2^* D_2$ is of the form $f_1^* s_1 + f_2^* s_2$ for $s_1 \in H^0 (C_1,D_1)$ and $s_2$ the unique element of $H^0 (C_2,D_2)$.  Thus, the divisor classes $[D_1], [D_2]$ are determined by $[f_1^*D_1 + f_2^* D_2]$, hence the map is generically injective.
\end{proof}

\subsection{Prym-Brill-Noether Theory of $\mathcal{R}_{B,g,0}$}
\label{Sec:PBNGenus1}

In this section, we focus on the case where $g(C_1) = 1$.  Our goal is to prove Proposition~\ref{Prop:BiellipticGenus1}, which identifies the Prym-Brill-Noether varieties $V^r(C,f)$ with the varieties $Y_{r+1}$.  This was previously known in the case $r=2$ \cite[Corollary~5.7]{Horing}.  One containment is straightforward.


\begin{lemma}
\label{Lem:PBNContains}
We have $Y_{r+1} \subseteq V^r (C,f)$.
\end{lemma}

\begin{proof}
Let $([D_1],[D_2]) \in Z_{r+1}$.  Since $h^0 (C_2 , D_2 ) \geq 1$ and pullback of sections is injective, we have the following inclusions:
\[
H^0 (C_1, D_1) \hookrightarrow H^0 (\WC, f_1^* D_1) \hookrightarrow H^0 (\WC, f_1^* D_1 + f_2^* D_2).
\]
Hence,
\[
h^0 (C_1 , D_1) \leq h^0 (\WC, f_1^* D_1 + f_2^* D_2).
\]
Since $\deg (D_1) = g(C_1)+r$, by Riemann-Roch we have $h^0 (C_1 , D_1 ) \geq r+1$.  It follows that $h^0 (\WC, f_1^* D_1 + f_2^* D_2) \geq r+1$, hence $f_1^* [D_1] + f_2^* [D_2] \in V^r (C,f)$.
\end{proof}

For the reverse containment, we need a preliminary lemma.

\begin{lemma}
\label{Lem:Containment}
Let $([D_1],[D_2]) \in Z_a$.  If $[D_1]$ is not $p_1$-simple, then $f_1^* [D_1] + f_2^* [D_2] \in Y_a \cap Y_{a-2}$.  Similarly, if $[D_2]$ is not $p_2$-simple, then $f_1^* [D_1] + f_2^* [D_2] \in Y_a \cap Y_{a+2}$. 
\end{lemma}

\begin{proof}
The divisor class $f_1^* [D_1] + f_2^* [D_2]$ is in $Y_a$ by definition.  We consider the case where $[D_1]$ is not $p_1$-simple.  The case where $[D_2]$ is not $p_2$-simple follows by a similar argument.  If $[D_1]$ is not $p_1$-simple, then by definition, there exists a point $x \in E$ such that $D_1 - p_1^* (x)$ is equivalent to an effective divisor $A$.  It follows that
\[
f_1^* D_1 + f_2^* D_2 \sim (f_1^* A + f_1^* p_1^* x) + f_2^* D_2 = f_1^* A + (f_2^* p_2^* x + f_2^* D_2) .
\]
The divisor class on the righthand side is in $Y_{a-2}$, and the result follows.
\end{proof}

\begin{corollary}
\label{Cor:WhenContains}
If $a \geq 1$, then $Y_{a+2} \subseteq Y_a$.  Moreover, if $g(C_1)=1$, for $([D_1],[D_2]) \in Z_a$, one has $f_1^* [D_1] + f_2^* [D_2] \in Y_{a+2}$ if and only if $[D_2]$ is not $p_2$-simple.
\end{corollary}

\begin{proof}
Let $([D_1],[D_2]) \in Z_{a+2}$.  Since $\deg (D_1) = g(C_1)+1+a \geq g(C_1)+2$, by Riemann-Roch we have $h^0 (C_1 , D_1 ) \geq 3$.  It follows that, for any effective divisor $D'$ of degree 2 on $C_1$, $D_1 - D'$ is equivalent to an effective divisor.  In particular, for any point $x \in E$, $D_1 - p_1^* (x)$ is equivalent to an effective divisor.  In other words, $[D_1]$ is not $p_1$-simple, so by Lemma~\ref{Lem:Containment}, we have $f_1^* [D_1] + f_2^* [D_2] \in Y_a$.  Since $([D_1],[D_2]) \in Z_{a+2}$ is arbitrary, we see that $Y_{a+2} \subseteq Y_a$.

Now, let $g(C_1) = 1$ and let $([D_1],[D_2]) \in Z_a$.  By Lemma~\ref{Lem:Containment}, if $[D_2]$ is not $p_2$-simple, then $f_1^* [D_1] + f_2^* [D_2] \in Y_{a+2}$.  If $[D_2]$ is $p_2$-simple, then by Lemma~\ref{Lem:Simple}, we have
\[
h^0 (\WC, f_1^* D_1 + f_2^* D_2) \leq 2h^0 (C_1, D_1) - a = a.
\]
By Lemma~\ref{Lem:PBNContains}, however, if $f_1^* [D_1] + f_2^* [D_2] \in Y_{a+2}$, then $h^0 (\WC, f_1^* D_1 + f_2^* D_2) \geq a+2$, a contradiction.
\end{proof}

We now prove Proposition~\ref{Prop:BiellipticGenus1}.

\begin{proof}[Proof of Proposition~\ref{Prop:BiellipticGenus1}]
Let $g(C_1)=1$.  By Lemma~\ref{Lem:PBNContains}, we have $Y_{r+1} \subseteq V^r (C,f)$.  We now prove the reverse containment.

Let $[D] \in V^r (C,f)$.  If $r$ is odd, then by \cite[Proposition~5.2.1]{Debarre}, we have $[D] \in Y_0$.  Similarly, if $r$ is even, then by \cite[Corollary~5.7]{Horing}, we have $[D] \in Y_3$.  Let $a$ be the maximum integer such that $[D] \in Y_a$.  Then there exists $([D_1],[D_2]) \in Z_a$ such that $D \sim f_1^* D_1 + f_2^* D_2$.  If $[D_2]$ is not $p_2$-simple, then by Lemma~\ref{Lem:Containment}, we have $[D] = [f_1^* D_1 + f_2^* D_2] \in Y_{a+2}$, contradicting the maximality of $a$.  It follows that $[D_2]$ is $p_2$-simple, hence by Lemma~\ref{Lem:Simple}, we have
\[
h^0 (\WC, D) = h^0 (\WC, f_1^* D_1 + f_2^* D_2) \leq 2h^0 (C_1, D_1) - a.
\]
By Riemann-Roch, the expression on the right is at most 2 if $a=0$ and equal to $a$ if $a \geq 1$.  Since $r \geq 2$ and $[D] \in V^r (C,f)$, it follows that $a \geq r+1$.  By Corollary~\ref{Cor:WhenContains}, we then have $[D] \in Y_{r+1}$.
\end{proof}

For our lifting theorem, we will need the following weak result about the intersection theory of $Z_a$.  We will make a more precise class computation in Corollary~\ref{Cor:Class} below.  Note that we do not claim that the intersection is reduced.  We thank Philip Engel for suggesting the degeneration argument in the proof of Lemma~\ref{Lem:Intersection}.

\begin{lemma}
\label{Lem:Intersection}
Let $g(C_1) = 1$.  For $r \geq 2$, there exists an integer $m_r \geq 2$ such that
\[
[ V^r (C,f) ] =  m_r \cdot \frac{\Theta^{r+1}}{(r+1)!} .
\]
Moreover, the intersection of $V^r (C,f)$ with $g-r-2$ general translates of $V^1 (C,f)$ consists of at least 
\[
2 \cdot \frac{(g-1)!}{(r+1)!}
\]
distinct points.
\end{lemma}

\begin{proof}
By \cite[Theorem~4.1(i)]{Debarre}, if $X$ is a hyperelliptic curve of genus $g-1$, then the Jacobian $\Jac (X)$ is contained in the closure of $\mathcal{R}_{B,g,0}$ in the moduli space of abelian varieties.  A general hyperelliptic curve has the property that the cohomology class of every subvariety of its Jacobian is an integral multiple of the minimal class $\frac{\Theta^k}{k!}$.  Thus, there exists an integer $m_r$ such that
\[
[ Y_{r+1} ] = m_r \cdot \frac{\Theta^{r+1}}{(r+1)!}.
\]
Similarly, there exists an integer $n_r$ such that the class of the reduced scheme $Y_{r+1}^{\mathrm{red}}$ is equal to $m_r \cdot \frac{\Theta^{r+1}}{(r+1)!}$.  It suffices to show that $m_r$ and $n_r$ are both at least 2.  By Corollary~\ref{Cor:Irreducible}, $Y_{g-1}$ consists of two distinct points, so the statement hold when $r+1 = g-1$.   When $r+1 \leq g-2$, the assumption that $r \geq 2$ implies that $g \geq 5$.  It therefore suffices to show that, for $g \geq 5$ and $d$ in the range $1 \leq d \leq g-2$, the Prym variety $\mathrm{Prym}(C,f)$ of a general cover $[f \colon \WC \to C] \in \mathcal{R}_{B,g,0}$ does not contain a closed subscheme $X$ of dimension $d$ with $[X] = \frac{\Theta^{g-1-d}}{(g-1-d)!}$.

We prove this by induction on $d$.  The base case, where $d=1$, follows from the Matsusaka-Ran criterion \cite[Theorem~3]{Ran} because for a general cover $[f \colon \WC \to C] \in \mathcal{R}_{B,g,0}$, the Prym variety $\mathrm{Prym}(C,f)$ is not a Jacobian.  For the inductive step, let $[f \colon \widetilde{D} \to D] \in \mathcal{R}_{B,g-1,0}$.  Let $E$ be a curve of genus 1, let $C_0$ be the 2-component curve obtained by identifying a point $p \in D$ with a point $q \in E$, and let $\WC_0$ be the 3-component curve obtained by identifying each of the 2 preimage points of $p$ in $\widetilde{D}$ with the point $q$ in two separate copies of $E$.  The admissible cover $[g \colon \WC_0 \to C_0]$ is contained in the closure of $\mathcal{R}_{B,g,0}$ in the moduli space $\overline{\mathcal{R}}_g$.  The Prym variety of this cover is isomorphic to $\mathrm{Prym}(D,f) \times \Jac(E)$.  Thus, if $X$ is a closed subscheme of this Prym variety with $[X] = \frac{\Theta^{g-1-d}}{(g-1-d)!}$, then the intersection of $X$ with $\mathrm{Prym}(D,f) \times \{ 0 \}$ has class $\frac{\Theta^{g-1-d}}{(g-1-d)!}$ as well.  By our induction hypothesis, if $[f \colon \widetilde{D} \to D]$ is sufficiently general, then $\mathrm{Prym}(D,f)$ does not contain a closed subscheme with this class, hence $\mathrm{Prym} (C_0,g)$ cannot contain such a closed subscheme either.
\end{proof}

\section{Prym Divisors on a Bielliptic Loop of Loops}
\label{Sec:LOL}

\subsection{Divisor Theory of Graphs}
\label{Sec:Divisors}

We briefly recall the divisor theory of metric graphs.  The tropical analogue of an algebraic curve is a metric graph.  A \emph{divisor} on a metric graph $\Gamma$ is a finite formal $\Z$-linear combination of points on $\Gamma$.  A divisor $D = \sum D(p) \cdot p$ is \emph{effective}, denoted $D \geq 0$, if $D(p) \geq 0$ for all $p \in \Gamma$.  We often refer to the integer $D(p)$ as the \emph{number of chips} of $D$ at $p$.

The tropical analogue of a rational function on a metric graph $\Gamma$ is a continuous, piecewise-linear function $\varphi \colon \Gamma \to \R$ with integer slopes.  The \emph{order of vanishing} of such a function $\varphi$ at a point $p \in \Gamma$, denoted $\ord_p (\varphi)$, is the sum of the incoming slopes of $\varphi$ at $p$.  The \emph{principal divisor} of $\varphi$ is the divisor
\[
\ddiv (\varphi) := \sum_{p \in \Gamma} \ord_p (\varphi) \cdot p .
\]
Two divisors $D, D'$ on $\Gamma$ are \emph{linearly equivalent}, denoted $D \sim D'$, if their difference is principal.  We write $[D]$ for the linear equivalence class of a divisor $D$.  The \emph{Picard group} of $\Gamma$, denoted $\Pic (\Gamma)$, is the group of linear equivalence classes of divisors.  The \emph{Baker-Norine rank} (or simply \emph{rank}) of a divisor $D$ is equal to $-1$ if $D$ is not equivalent to an effective divisor, and otherwise equal to the maximum integer $r$ such that, for all effective divisors $E$ of degree $r$, $D-E$ is equivalent to an effective divisor.

A \emph{free double cover} of a metric graph $\Gamma$ is a degree 2 covering space $\pi \colon \WG \to \Gamma$ such that the restriction of $\pi$ to each edge in $\WG$ is an isometry onto its image.  A divisor class $[D] \in \Pic (\WG)$ is a \emph{Prym divisor class} if it has canonical norm.  The set of Prym divisor classes is denoted $P (\Gamma,\pi)$.  The \emph{tropical Prym-Brill-Noether variety} is defined to be:
\[
V^r (\Gamma,\pi) := \{ [D] \in P(\Gamma,\pi) \mid \rk (D) \geq r \text{ and } \rk (D) \equiv r \Mod{2} \}.
\]

\subsection{The loop of loops}
\label{Sec:TheLOL}

Let $\Gamma$ be a metric graph with vertices $v_i$, $w_i$ for $i \in \Z/(g-1)\Z$, with 1 edge connecting $w_{i-1}$ to $v_i$ and 2 edges connecting $v_i$ to $w_i$ for all $i$.  See Figure~\ref{Fig:LOL}.  The metric graph $\Gamma$ is called a \emph{loop of loops of genus} $g$.  The union of the two edges connecting $v_i$ to $w_i$ is called the $i$th loop, and denoted $\gamma_i$.  If the two edges comprising $\gamma_i$ have the same length for all $i$, we say that $\Gamma$ is \emph{bielliptic} or 2-\emph{uniform}.  When $\Gamma$ is bielliptic, there exists a harmonic morphism $\psi$ of degree 2 from $\Gamma$ to a single loop $\Sigma$, as pictured in Figure~\ref{Fig:Bielliptic}.  

\begin{figure}[H]
\begin{tikzpicture}

\draw (0,0) circle (2);
\draw [fill=white] (0,2) circle (0.5);
\draw (0.65,2) node {$v_2$};
\draw [fill=white] (1.7,1) circle (0.5);
\draw (1.7,0.35) node {$v_1$};
\draw (1.5,1.6) node {$w_1$};
\draw [fill=white] (1.7,-1) circle (0.5);
\draw (1.7,-0.35) node {$w_6$};
\draw [fill=white] (0,-2) circle (0.5);
\draw [fill=white] (-1.7,-1) circle (0.5);
\draw [fill=white] (-1.7,1) circle (0.5);

\end{tikzpicture}
\caption{A loop of loops of genus 7.}
\label{Fig:LOL}
\end{figure}
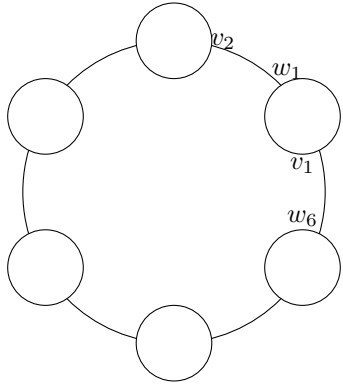

\begin{figure}[H]
\begin{tikzpicture}

\draw (0,0) circle (2);
\draw [fill=white] (0,2) circle (0.5);
\draw [fill=white] (1.7,1) circle (0.5);
\draw [fill=white] (1.7,-1) circle (0.5);
\draw [fill=white] (0,-2) circle (0.5);
\draw [fill=white] (-1.7,-1) circle (0.5);
\draw [fill=white] (-1.7,1) circle (0.5);
\draw (0,0) circle (0.75);
\draw (1.75,0) edge[->] (1,0);
\draw (0.87,1.48) edge[->] (0.5,0.85);
\draw (-0.87,1.48) edge[->] (-0.5,0.85);
\draw (-1.75,0) edge[->] (-1,0);
\draw (-0.87,-1.48) edge[->] (-0.5,-0.85);
\draw (0.87,-1.48) edge[->] (0.5,-0.85);

\end{tikzpicture}
\caption{The harmonic morpshim $\psi \colon \Gamma \to \Sigma$.}
\label{Fig:Bielliptic}
\end{figure}
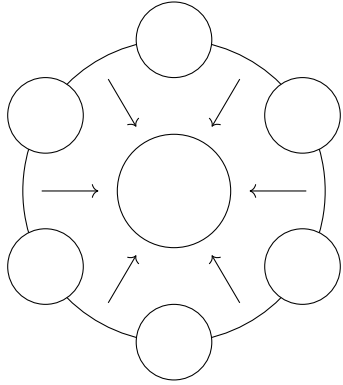

Let $\WG$ be a bielliptic loop of loops of genus $2g-1$ such that, for all $i$, the edge from $w_{i-1}$ to $v_i$ has the same length as the edge from $w_{g+i-2}$ to $v_{g+i-1}$, and the circumference of $\gamma_i$ is the same as that of $\gamma_{g-1+i}$.  Then $\WG$ admits an antipodal involution $\iota$, and the quotient $\Gamma = \WG / \langle \iota \rangle$ is a bielliptic loop of loops of genus $g$.  We write $\pi \colon \WG \to \Gamma$ for the quotient map.  Since $\WG$ is itself bielliptic, there exists a degree 2 harmonic morphism $\pi_1 \colon \WG \to \widetilde{\Sigma}$ to a single loop.  The composition of the antipodal involution $\iota$ with the involution $\iota_1$ exchanging the sheets of $\pi_1$ is a third involution $\iota_2$, and the quotient $\WG / \langle \iota_2 \rangle$ is isomorphic to $\Gamma$.  Together, these maps fit into a commutative diagram:
\begin{equation}
\label{eq:tropical}
\xymatrix {
& \WG \ar[ld]_{\pi_1} \ar[d]^{\pi} \ar[rd]^{\pi_2} & \\
\widetilde{\Sigma} \ar[rd]_{\psi_1} & \Gamma \ar[d]^{\psi} & \Gamma \ar[ld]^{\psi_2} \\
& \Sigma & 
 }
\end{equation}

In \cite{PBN}, the fourth author provides a complete description of Prym divisors on $\WG$.  First, let $\mathbb{T}^0$ be the set of effective divisors of degree $2g-2$ on $\WG$ that have exactly 1 chip on each of the loops $\gamma_i$.  If a divisor class $[D] \in \Pic (\WG)$ has a representative in $\mathbb{T}^0$, then this representative is unique, and we call it the \emph{orientable representative}.  Let $P^0 \subset \Pic (\WG)$ be the subset of divisor classes that have a representative in $\mathbb{T}^0$ and such that the clockwise distance from $v_i$ to the chip on $\gamma_i$ is equal to the counterclockwise distance from $w_{g+i-1}$ to the chip on $\gamma_{g+i-1}$.  Fix a basepoint $p_0 \in \WG$ on the edge from $w_{2g-2}$ to $v_1$, and define $P^1 = \{ [D-p_0 + \iota^* p_0 ] \} \mid [D] \in P^0 \}$.

\begin{proposition} \cite[Proposition~3.3]{PBN}
The set of Prym divisor classes on $\WG$ is the disjoint union of $P^0$ and $P^1$.
\end{proposition}

The earlier paper \cite{PBN} also describes the Prym-Brill-Noether variety of the cover $\pi \colon \WG \to \Gamma$.  When $\Gamma$ is bielliptic, the exposition can be simplified significantly.  Given a positive integer $r$ and elements $i \in \Z/(2g-2)\Z, j \in \Z/2\Z$, we define the point:
\[
x^r_{i,j} := \left\{ \begin{array}{ll}
w_i &\text{if } r \text{ is even and }j \equiv \frac{r}{2} \Mod{2}, \text{ or } \\
& r \text{ is odd, } j \in \{ 1, \ldots , g-1 \}, \text{ and } j \equiv \frac{r+1}{2} \Mod{2}, \text{ or } \\
& r \text{ is odd, } j \in \{ g, \ldots , 2g-2 \}, \text{ and } j \equiv \frac{r-1}{2} \Mod{2} \\
v_i &\text{otherwise. }
\end{array} \right.
\]

The Coxeter group $I_2 (r+1)$ is given in terms of generators and relations as
\[
W = \langle r_0, r_1 \mid r_0^2 = r_1^2 = (r_0 r_1)^{r+1} = \epsilon \rangle ,
\]
where $\epsilon$ denotes the identity in the group.  The group $W$ is isomorphic to the dihedral group of order $2r+2$.  In this Coxeter group, there are exactly 2 reduced words of maximal length $r+1$.  Specifically, there is an alternating sequence of the two generators that begins with $r_0$, and an alternating sequence of the two generators that begins with $r_1$.  A \emph{lingering word} of length $n$ in this Coxeter group is a function $s \colon [n] \to \{ r_0 , r_1 , \epsilon \}$, typically written $Q = s_1 s_2 \cdots s_n$.  Given a lingering word $Q$ of length $g-1$, we define
\[
T(Q) := \{ [D] \in P^r \mid \text{ if } s_i = r_j, \text{ then the chip of the orientable representative of } D \text{ on } \gamma_i \text{ is at } x^r_{i,j} \} .
\]
By \cite[Lemma~4.7]{PBN}, the set $T(Q)$ is a real torus of dimension $\# \{ i \mid s_i = \epsilon \}$.

\begin{theorem} \cite[Theorem~6.8]{PBN}
\label{Thm:PBNBielliptic}
Let $\Gamma$ be a bielliptic loop of loops and $\pi \colon \WG \to \Gamma$ the quotient by the antipodal involution.  If $r \geq 2$, then
\[
V^r (\Gamma,\pi) = \bigcup T(Q),
\]
where the union is over all lingering words of length $g-1$ in the Coxeter group $I_2 (r+1)$ that contain one of the two reduced subwords of maximal length $r+1$.

In the special case where $r=2$, the statement holds for every loop of loops $\Gamma$, regardless of edge lengths.
\end{theorem}

In other words, the union is over all lingering words $Q$ of length $g-1$ that contain a length-$(r+1)$ alternating subword of $s_0$'s and $s_1$'s.  In \cite{PBN}, the 4th author defines the \emph{lingering subword poset}, whose elements are precisely such words.  Given a Coxeter group $W$, an element $w_0 \in W$, and a positive integer $n$, there is a lingering subword poset $P(W, w_0 , n)$.  For our purposes, the Coxeter group will always be $I_2 (r+1)$ and the element $w_0$ will be the maximal element with respect to Bruhat order, so for ease of notation, we write $P(r+1,n) = P(I_2 (r+1), w_0, n)$.  The elements of $P(r+1,n)$ are lingering words of length $n$ that contain a length-$(r+1)$ alternating subword of $s_0$'s and $s_1$'s.  Given two elements $Q = s_1 \cdots s_n, Q'=s'_1 \cdots s'_n \in P(r+1,n)$, we say that $Q \leq Q'$ if, for all $i$, either $s'_i = s_i$ or $s'_i = \epsilon$.  This poset encodes the containment relations among the tori $T(Q)$.

\begin{lemma}\cite[Lemma~4.8 and Proposition~4.20]{PBN}
\label{Lem:Contains}
For $Q, Q' \in P(r+1, g-1)$, we have $T(Q) \subseteq T(Q')$ if and only if $Q \leq Q'$.  Moreover, the intersection $T(Q) \cap T(Q')$ is nonempty if and only if $Q$ and $Q'$ have a common lower bound, in which case $T(Q) \cap T(Q') = T(Q \wedge Q')$.
\end{lemma}

\subsection{Zonotopes and Oriented Matroids}
\label{Sec:Zonotope}

In this section, we show that the intersection lattice of the tori in Theorem~\ref{Thm:PBNBielliptic} is the face lattice of a certain polytope called a zonotope.  Given vectors $\vec{x}_1 , \ldots , \vec{x}_n \in \mathbb{R}^d$, the corresponding \emph{zonotope} is the Minkowski sum:
\[
Z (\vec{x}_1 , \ldots , \vec{x}_n) := \{ a_1 \vec{x}_1 + \cdots a_n \vec{x}_n \mid 0 \leq a_i \leq 1 \text{ for all } i \} .
\]

A set of vectors in $\mathbb{R}^d$ also defines an oriented matroid.  A \emph{signed subset} of a finite set $E$ is a pair $C = (C^+, C^-)$ of disjoint subsets of $E$.  An \emph{oriented matroid} with ground set $E$ is a collection of signed subsets, called \emph{signed circuits}, satisfying the following axioms:
\begin{enumerate}
\item  If $C = (C^+, C^- )$ is a signed circuit, then $C^+ \cup C^- \neq \emptyset$.
\item  If $C = (C^+, C^-)$ is a signed circuit, then so is $-C = (C^-, C^+ )$. 
\item  No proper subset of a signed circuit is a signed circuit.
\item  If $C_0, C_1$ are signed circuits with $C_1 \neq -C_0$ and $e \in C_0^+ \cap C_1^-$, then there exists a signed circuit $C$ with $C^+ \subseteq (C_0^+ \cup C_1^+) \smallsetminus \{ e \}$ and $C^- \subseteq (C_0^- \cup C_1^- ) \smallsetminus \{ e \}$.
\end{enumerate}

Let $\vec{x}_1 , \ldots , \vec{x}_n \in \mathbb{R}^d$.  One can define the corresponding oriented matroid as follows.  Each minimal linear dependence is of the form $\sum \lambda_i \vec{x} = 0$ for some scalars $\lambda_i \in \mathbb{R}$.  Given such a linear dependence, we let $C^+ = \{ i \in [n] \mid \lambda_i > 0 \}$ and $C^- = \{ i \in [n] \mid \lambda_i < 0 \}$.  The signed sets $C = (C^+ , C^-)$ are the signed circuits of an oriented matroid.  If one performs the same construction for all linear dependences, rather than just the minimal ones, one obtains the signed \emph{vectors} of an oriented matroid.  An oriented matroid is called \emph{realizable} if it corresponds to a collection of vectors in $\mathbb{R}^d$ in this way.

The key fact relating zonotopes and oriented matroids is the following.

\begin{proposition} \cite[Proposition~2.2.2]{OrientedMatroids}
\label{Prop:ZonotopeMatroid}
Let $\vec{x}_1 , \ldots , \vec{x}_n \in \mathbb{R}^d$ and let $M$ be the corresponding oriented matroid.  There is an order-reversing bijection between the face lattice of the zonotope $Z (\vec{x}_1 , \ldots , \vec{x}_n)$ and the lattice of covectors of $M$.
\end{proposition}

We are now ready to prove the main result of this section.

\begin{theorem}
\label{Thm:Zonotope}
The poset $P(r+1,n)$ is the face lattice of a zonotope.
\end{theorem}

\begin{proof}
Fix positive integers $r$ and $n$ with $r+1 \leq n$.  There exists an oriented matroid on $E=[n]$ of rank $n-r$, called the \emph{alternating matroid}, with signed circuits defined as follows.   For every subset $C \subseteq E$ of size $\vert C \vert = r+1$, write $C = \{ c_1, c_2, \ldots, c_{r+1} \}$ with $c_i < c_{i+1}$ for all $i$.  Then $(C^+, C^-)$ is a signed circuit of the alternating matroid, where $C^+ = \{ c_1, c_3, c_5, \ldots \}$ and $C^- = \{ c_2 , c_4, c_6 , \ldots \}$.  The set of signed vectors of the alternating matroid are the signed sets $V=(V^+,V^-)$ such that there exists a signed circuit $C=(C^+,C^-)$ with $C^+ \subseteq V^+$ and $C^- \subseteq V^-$.  Given two signed vectors $V_1 = (V_1^+ , V_1^-), V_2 = (V_2^+ , V_2^-)$, we have $V_1 \leq V_2$ if and only if $V_1^+ \subseteq V_2^+$ and $V_1^- \subseteq V_2^-)$.

Given a signed vector $V = (V^+ , V^-)$ in the alternating matroid, we define a lingering word $Q = s_1 \cdots s_n$ by
\[
s_i = \left\{ \begin{array}{ll}
r_0 & \text{if } i \in V^+ \\
r_1 & \text{if } i \in V^- \\
\epsilon &\text{otherwise. }
\end{array} \right.
\]
This defines an order-reversing bijection from the lattice of vectors of the alternating matroid to $P(r+1,n)$.

By \cite[Corollary~8.2.10]{OrientedMatroids}, the alternating matroid is realizable.  An oriented matroid is realizable if and only if its dual matroid is realizable.  Thus, there exists a set of vectors $\vec{x}_1 , \ldots , \vec{x}_n \in \mathbb{R}^{n-r}$ whose lattice of covectors is anti-isomorphic to $P(r+1,n)$.  The result then follows from Proposition~\ref{Prop:ZonotopeMatroid}.
\end{proof}

\begin{theorem}
\label{Thm:PrymZonotope}
Let $\Gamma$ be a bielliptic loop of loops and $\pi \colon \WG \to \Gamma$ the quotient by the antipodal involution.  If $r \geq 2$, then $V^r (\Gamma,\pi)$ is a union of tori, and the intersection lattice of these tori is the face lattice of a zonotope. 

In the special case where $r=2$, the statement holds for every loop of loops $\Gamma$, regardless of edge lengths.
\end{theorem}

\begin{proof}
Theorem~\ref{Thm:PBNBielliptic} shows that $V^r (\Gamma,\pi)$ is a union of tori, and Lemma~\ref{Lem:Contains} shows that the intersection lattice of these tori is $P(r+1,g-1)$.  The result then follows from Theorem~\ref{Thm:Zonotope}.
\end{proof}

\begin{example}
\label{Ex:LowDim}
A consequence of Theorem~\ref{Thm:Zonotope} is that the order complex of $P(r+1,g-1)$ is homeomorphic to a sphere of dimension $g-r-2$.  Thus, when $g-r-2=0$, the tropical Prym-Brill-Noether variety $V^r (\Gamma,\pi)$ consists of 2 points.  When $g-r-2=1$, the poset $P(r+1,g-1)$ is the face poset of a graph with $2g-2$ vertices and $2g-2$ edges, arranged in a cycle.  Thus, in this case, the Prym-Brill-Noether variety $V^r (\Gamma,\pi)$ is itself a loop of loops of genus $2g-1$.  One can check that it is isomorphic to the metric graph obtained from $\WG$ by contracting the edge from $w_{i-1}$ to $v_i$ for all $i$.  (The Prym-Brill-Noether variety of the cover $\pi \colon \WG \to \Gamma$ is independent of the lengths of these edges, so one cannot expect it to be isomorphic to $\WG$ itself.)  By Lemma~\ref{Lem:Irreducible}, this is analogous to the classical case.  Specifically, when $g(C_1)=1$ and $g-r-2=0$, the classical Prym-Brill-Noether variety $V^r (C,f)$ consists of 2 points, and when $g-r-2=1$, the classical Prym-Brill-Noether variety is isomorphic to $\WC$.
\end{example}

\begin{example}
\label{Ex:RhombicIcosahedron}
We illustrate Theorem~\ref{Thm:PrymZonotope} in the first case where the Prym-Brill-Noether variety is two-dimensional, when $g=6$ and $r=2$.  In this case, $V^2 (\Gamma,\pi)$ is a union of tori $T(Q)$, each of dimension 2, 1, or 0.  The poset $P(3,5)$ is the face lattice of the \emph{rhombic icosahedron}, pictured in Figure~\ref{Fig:RhombicIcosahedron}.  In the figure, each of the 20 2-dimensional faces is labeled with a lingering word $Q$ in the symbols 0, 1, and $\epsilon$, corresponding to the 2-dimensional torus $T(Q)$.  When the faces labeled $Q$ and $Q'$ meet along an edge, it means that the intersection $T(Q) \cap T(Q')$ is a circle.  Similarly, the 22 vertices, which correspond to the 22 words of length 5 in the symbols 0 and 1 that contain a length-3 alternating subword, also correspond to points where two or more of the 1-dimensional tori intersect.

 \begin{figure}[h]
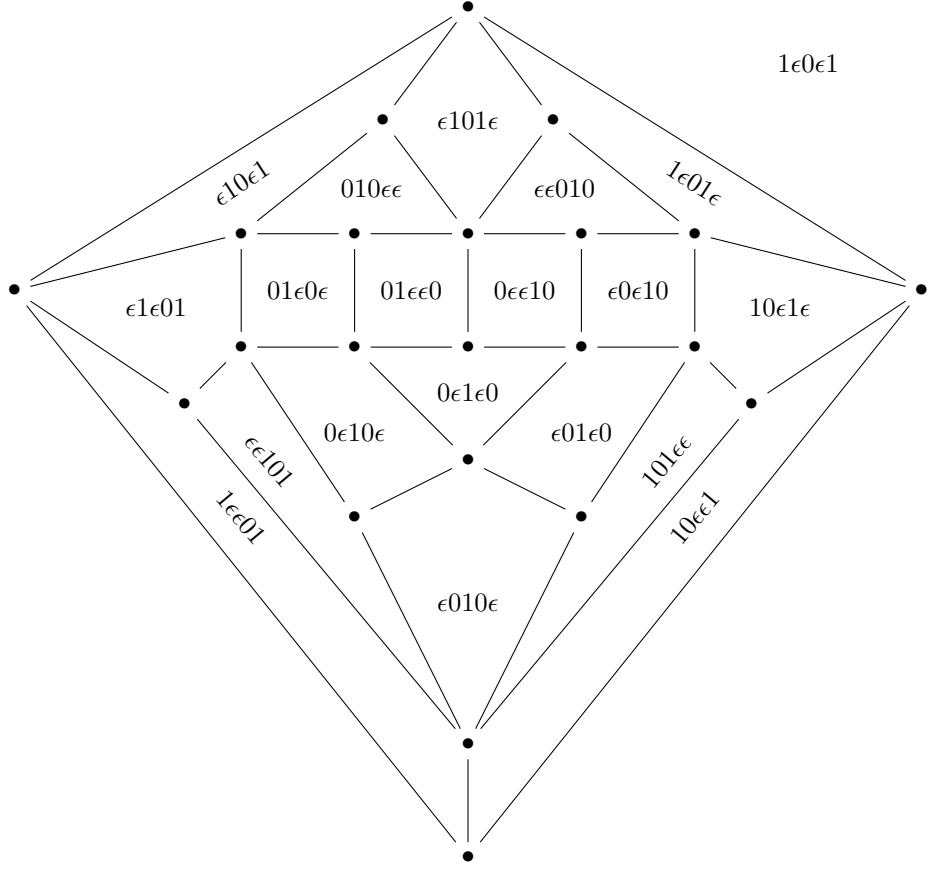

\begin{center}

\tikz[scale=0.75]{
 \node (v1) at (0,0)   {$\bullet$};
  \node (v2) at (-1.5,-2) {$\bullet$};
  \node (v3) at (1.5,-2)  {$\bullet$};
  \node (v4) at (0,-4)  {$\bullet$};
\node (v5) at (-8,-5)  {$\bullet$};
\node (v6) at (-4,-4)  {$\bullet$};
\node (v7) at (-2,-4)  {$\bullet$};
\node (v8) at (2,-4)  {$\bullet$};
\node (v9) at (4,-4)  {$\bullet$};
\node (v10) at (8,-5)  {$\bullet$};
\node (v11) at (-4,-6)  {$\bullet$};
\node (v12) at (-2,-6)  {$\bullet$};
\node (v13) at (0,-6)  {$\bullet$}; 
\node (v14) at (2,-6)  {$\bullet$};
\node (v15) at (4,-6)  {$\bullet$};
\node (v16) at (-5,-7)  {$\bullet$};
\node (v17) at (-2,-9)  {$\bullet$};
\node (v18) at (0,-8)  {$\bullet$};
\node (v19) at (2,-9)  {$\bullet$};
\node (v20) at (5,-7)  {$\bullet$};
\node (v21) at (0,-13)  {$\bullet$};
\node (v22) at (0,-15)  {$\bullet$};
  \graph { (v1) -- (v2) -- (v4) -- (v3) -- (v1)
   };
   \graph { (v1) -- (v5) -- (v22) -- (v10) -- (v1)
   };
  \graph { (v2) -- (v6) -- (v7) -- (v4) -- (v8) -- (v9) -- (v3)
   };
  \graph { (v6) -- (v5) -- (v16) -- (v11) -- (v12) -- (v13) -- (v14) -- (v15) -- (v20) -- (v10) -- (v9)
   };
  \graph { (v11) -- (v17) -- (v18) -- (v12)
   };
  \graph { (v15) -- (v19) -- (v18) -- (v14)
   };
  \graph { (v17) -- (v21) -- (v19)
   };
  \graph { (v16) -- (v21) -- (v20)
   };
  \graph { (v11) -- (v6)
   };
  \graph { (v12) -- (v7)
   };
  \graph { (v13) -- (v4)
   };
  \graph { (v14) -- (v8)
   };
  \graph { (v15) -- (v9)
   };
  \graph { (v21) -- (v22)
   };

   \node (f1) at (0,-2) {$\epsilon101\epsilon$};
   \node[rotate = 33] (f2) at (-4,-3.1) {$\epsilon10\epsilon1$};
   \node[rotate = 327] (f3) at (4,-3.1) {$1\epsilon01\epsilon$};
   \node (f4) at (-1.7,-3.2) {$010\epsilon\epsilon$};
   \node (f5) at (1.7,-3.2) {$\epsilon\epsilon010$};
   \node (f6) at (-5.5,-5.3) {$\epsilon1\epsilon01$};
   \node (f7) at (-3,-5) {$01\epsilon0\epsilon$};
   \node (f8) at (-1,-5) {$01\epsilon\epsilon0$};
   \node (f9) at (1,-5) {$0\epsilon\epsilon10$};
   \node (f10) at (3,-5) {$\epsilon0\epsilon10$};
   \node (f11) at (5.5,-5.3) {$10\epsilon1\epsilon$};
   \node[rotate = 310] (f12) at (-4,-9) {$1\epsilon\epsilon01$};
   \node[rotate = 50] (f13) at (4,-9) {$10\epsilon\epsilon1$};
   \node[rotate = 310] (f14) at (-3.5,-8) {$\epsilon\epsilon101$};
    \node[rotate = 50] (f15) at (3.5,-8) {$101\epsilon\epsilon$};
    \node (f16) at (-2,-7.5) {$0\epsilon10\epsilon$};
  \node (f17) at (2,-7.5) {$\epsilon01\epsilon0$};
\node (f18) at (0,-6.8) {$0\epsilon1\epsilon0$};
\node (f19) at (0,-10.5) {$\epsilon010\epsilon$};
\node (f20) at (6,-1) {$1\epsilon0\epsilon1$};

} 
\end{center}
\caption{A rhombic icosahedron with faces labeled by lingering words}
\label{Fig:RhombicIcosahedron}
\end{figure}
\end{example}

\section{Proof of the Lifting Theorem}
\label{Sec:Lifting}

In this section, we prove Theorem~\ref{Thm:Lifting}.  A definition of \emph{vertex avoiding} divisor classes $[D] \in V^r (\Gamma,\pi)$ is provided in \cite[Definition~5.2]{PBN}.  When $\Gamma$ is bielliptic, this definition can be simplified.  

\begin{definition}
\label{Def:VA}
Let $\Gamma$ be a bielliptic loop of loops, let $\pi \colon \WG \to \Gamma$ the quotient by the antipodal involution, and let $Q = s_1 \cdots s_{g-1} \in P(r+1,g-1)$ be a word with $\# \{ j \mid s_j = \epsilon \} = g-2-r$.  A divisor $[D] \in T(Q)$ is \emph{vertex avoiding} if, for all $i$ such that $s_i = \epsilon$, the chip of the orientable representative of $[D]$ on the $i$th loop is neither $v_i$ nor $w_i$.
\end{definition}

We first require a tropical analogue of Lemma~\ref{Lem:Intersection}.

\begin{lemma}
\label{Lem:TropIntersection}
Let $\Gamma$ be a bielliptic loop of loops and $\pi \colon \WG \to \Gamma$ the quotient by the antipodal involution.  For $r \geq 2$, there exist divisors $D_1 , \ldots , D_{g-2-r} \in P^r (\Gamma,\pi) - P^1 (\Gamma,\pi)$ such that
\[
X = \bigcap_{i=1}^{g-2-r} [V^1 (\Gamma,\pi) + D_i ] \cap V^r (\Gamma,\pi)
\]
is finite and
\[
\# X = 2 \cdot \frac{(g-1)!}{(r+1)!}.
\]
Moreover, if $[D] \in V^r (\Gamma,\pi)$ is vertex avoiding, then the divisors $D_i$ can be chosen to so that $[D] \in X$.
\end{lemma}

\begin{proof}
By Theorem~\ref{Thm:PBNBielliptic}, $V^r (\Gamma,\pi)$ is a union of tori $T(Q)$ corresponding to words in $P(r+1,g-1)$, and by Lemma~\ref{Lem:Contains}, the maximal tori with respect to containment correspond to words with $\# \{ j \mid s_j = \epsilon \} = g-2-r$.  Each such word is obtained by choosing a subset of $[g-1]$ of size $g-2-r$ to fill with $\epsilon$'s, and filling the remaining entries with one of the two alternating substrings of length $r+1$.  It follows that the number of maximal tori in this union is exactly $2 \cdot {{g-1}\choose{g-2-r}}$.

Now, $V^1 (\Gamma,\pi)$ is a union of translates of the $g-1$ codimension-1 coordinate tori.  Each maximal torus $T(Q)$ therefore intersects $g-2-r$ general translates of $V^1 (\Gamma,\pi)$ in $(g-2-r)!$ points.  It follows that the intersection of $V^r (\Gamma,\pi)$ with these translates has cardinality
\[
2 \cdot {{g-1}\choose{g-2-r}} \cdot (g-2-r)! = 2 \cdot \frac{(g-1)!}{(r+1)!}.
\]

These translates can be chosen to contain a particular point $[D] \in V^r (\Gamma,\pi)$ if $[D]$ lies in exactly one maximal torus $T(Q)$, which holds if $[D]$ is vertex avoiding.
\end{proof}

Next, we show that $V^r (\Gamma,\pi)$ is a local complete intersection of translates of $V^1 (\Gamma,\pi)$.

\begin{proposition}
\label{Prop:LCI}
Let $\Gamma$ be a bielliptic loop of loops and $\pi \colon \WG \to \Gamma$ the quotient by the antipodal involution.  Let $r \geq2$ and let $[D] \in V^r (\Gamma,\pi)$ be vertex avoiding.  There exists an open neighborhood $U$ of $[D]$ and effective divisors $E_1 , \ldots , E_{r+1}$ of degree $r-1$ such that
\[
V^r (\Gamma,\pi) \cap U = \Big( \bigcap_{i=1}^{r+1} V^1 (\Gamma,\pi) + E_i - \iota^* E_i \Big) \cap U .
\]
\end{proposition}

\begin{proof}
Let $[D'] \in P^r (\Gamma,\pi)$, and let $D' \in \vert D' \vert$ be its orientable representative.  Let $i \in [g-1]$ and let $p_i$ be a general point on the $i$th loop.  The key to this argument is to compute the orientable representative of $[D' - p_i + \iota^* p_i]$.  There are 4 cases:
\begin{enumerate}
\item  If $r$ is even, then the orientable representative of $[D' - p_i + \iota^* p_i]$ has a chip on the same point of the $j$th loop as $D'$ for all $j > i$, and a chip on the antipodal point of the $j$th loop from $D'$ for all $j<i$.  
\item  If $r$ is odd, then the orientable representative of $[D' - p_i + \iota^* p_i]$ has a chip on the same point of the $j$th loop as $D'$ for all $j < i$, and a chip on the antipodal point of the $j$th loop from $D'$ for all $j>i$.
\item  If $r$ is even, then the orientable representative of $[D' + p_i - \iota^* p_i]$ has a chip on the same point of the $j$th loop as $D'$ for all $j < i$, and a chip on the antipodal point of the $j$th loop from $D'$ for all $j>i$.  
\item  If $r$ is odd, then the orientable representative of $[D' + p_i - \iota^* p_i]$ has a chip on the same point of the $j$th loop as $D'$ for all $j > i$, and a chip on the antipodal point of the $j$th loop from $D'$ for all $j<i$.
\end{enumerate}

Now, let $Q = s_1 \cdots s_{g-1} \in P(r+1,g-1)$ and let $A = \{ i \in [g-1] \mid s_i \neq \epsilon \}$.  Suppose that $\# A = r+1$ and $[D] \in T(Q)$.  For $i \in A$, let $i'$ be either the largest element of $A$ that is smaller than $i$, or the smallest element of $A$ that is larger than $i$.  (Since $\# A \geq 2$, there always exists an element $i' \in A$ that is not equal to $i$.)  Now, partition $A \smallsetminus \{ i,i' \}$ into sets $B_i$ and $B'_i$ such that:
\begin{enumerate}
\item  if $r$ is even and the orientable representative $D$ of $[D]$ has a chip at $v_i$, or $r$ is odd and $D$ has a chip at $w_i$, then $\# \{ b \in B_i \mid b < i,i' \}$ is equal to either $\# \{ b \in B'_i \mid b < i,i' \}$ or $\# \{ b \in B'_i \mid b < i,i' \} + 1$ and $\# \{ b \in B_i \mid b > i,i' \}$ is equal to either $\# \{ b \in B'_i \mid b > i,i' \}$ or $\# \{ b \in B'_i \mid b > i,i' \} - 1$. 
\item  if $r$ is even and $D$ has a chip at $w_i$, or $r$ is odd and $D$ has a chip at $v_i$, then $\# \{ b \in B_i \mid b < i,i' \}$ is equal to either $\# \{ b \in B'_i \mid b < i,i' \}$ or $\# \{ b \in B'_i \mid b < i,i' \} - 1$ and $\# \{ b \in B_i \mid b > i,i' \}$ is equal to either $\# \{ b \in B'_i \mid b > i,i' \}$ or $\# \{ b \in B'_i \mid b > i,i' \} + 1$.
\end{enumerate}
Let
\[
E_i = \sum_{j \in B_i} p_j + \sum_{j \in B'_i} \iota^* p_j .
\]
Our goal is to show that
\[
V^r (\Gamma,\pi) \cap U = \Big( \bigcap_{i=1}^{r+1} V^1 (\Gamma,\pi) + E_i - \overline{E}_i \Big) \cap U .
\]
By the definition of rank, the lefthand side is contained in the righthand side.  It therefore suffices to show the reverse containment.

Let $[D'] \in (V^1 (\Gamma,\pi) + E_i - \iota^* E_i) \cap U$, and let $D'_i$ be the orientable representative of $[D' - E_i + \iota^* E_i]$.  Since $[D'] \in U$ and $U$ is sufficiently small, it follows by induction, using (1)-(4) above, that $D'_i$ has a chip in a small neighborhood of $v_i$, and this chip is at $v_i$ if and only if $D'$ has a chip on the same point of the $i$th loop as $D$.  Similarly, $D'_i$ has a chip in a small neighborhood of $w_{i'}$, and a chip on the $j$th loop at a point distinct from $v_j$ or $w_j$ for all $j \notin \{ i,i' \}$.  By \cite[Lemma~3.5]{PBN}, since $D'_i$ has rank at least one, there exists $j \in [g-1]$ such that $D'_i$  has a chip at $v_j$.  By the above, we see that $j$ must be equal to $i$, hence $D'$ has a chip on the same point of the $i$th loop as $D$.  

Now, if $[D'] \in \cap_{i=1}^{r+1} (V^1 (\Gamma,\pi) + E_i - \iota^* E_i) \cap U$, it follows that $D'$ has a chip on the same point of the $i$th loop as $D$ for all $i$ such that $s_i \neq \epsilon$.  In other words, $[D'] \in T(Q)$, and the result follows.
\end{proof}

Before proceeding to the proof of Theorem~\ref{Thm:Lifting}, we demonstrate that it is not vacuous.

\begin{lemma}
\label{Lem:Nonvacuous}
Let $\Gamma$ be a bielliptic loop of loops of genus $g$ and $\pi \colon \WG \to \Gamma$ the quotient by the antipodal involution.  Then there exists a bielliptic curve $C$ of genus $g$ and an \'{e}tale double cover $f \colon \WC \to C$ specializing to $\pi$.  Moreover, the diagram \eqref{eq:tropical} lifts to the diagram \eqref{eq:classical}.  In particular, $g (C_1) = 1$.
\end{lemma}

\begin{proof}
If the residue field has odd characteristic or characteristic zero, then by \cite[Lemma~7.15]{ABBR15a}, there is a map of smooth curves $p \colon C \to E$ specializing to $\psi$, where $C$ has genus $g$ and $E$ has genus 1.  By \cite[Lemma~5.9]{JensenLen18}, there exists a unique \'{e}tale double cover $f \colon \WC \to C$ specializing to $\pi$.

We show that the cover $p \circ f \colon \WC \to E$ is Galois.  Let $\tau \colon C \to C$ be the involution that transposes the sheets of the cover $p \colon C \to E$, and let $\eta \in \Jac_2 (C)$ be the unique nontrivial 2-torsion divisor in the kernel of $f^*$.  By \cite[Corollary~4.3]{Naranjo}, the cover is Galois if and only if $\tau^* (\eta) = \eta$.  The involution $\tau$ specializes to the involution of $\Gamma$ that transposes the sheets of the cover $\psi $.  Composing this involution with the cover $\pi \colon \WG \to \Gamma$ produces an isomorphic double cover.  Since the double cover lifts uniquely, it follows that $\tau^* (\eta) = \eta$.

Since the cover $p \circ f \colon \WC \to E$ is Galois, by \cite{Debarre} the Galois group must be the Klein-4 subgroup of $S_4$ that acts on the 4 sheets by products of 2 disjoint transpositions.  Each of these involutions specializes to an involution of $\WG$ that acts on the sheets in the same way.  In other words, the involutions $\iota , \iota_1$, and $\iota_2$ are specializations of involutions on $\WC$.   Thus, the diagram \eqref{eq:tropical} lifts to the diagram \eqref{eq:classical}.
\end{proof}

We now prove the main theorem.

\begin{proof}[Proof of Theorem~\ref{Thm:Lifting}]
The case where $r=1$ is proven in \cite[Lemma~5.1]{PBN}, so we assume that $r \geq 2$.  The tropicalization of the Prym-Brill-Noether variety $V^r (C,f)$ is a closed subset of $V^r (\Gamma,\pi)$, and the set of vertex avoiding divisor classes is dense in $V^r (\Gamma,\pi)$, so it suffices to consider the case where $[D]$ is vertex avoiding.  By Lemma~\ref{Lem:TropIntersection}, there exist divisors $D_1 , \ldots , D_{g-2-r} \in P^r (\Gamma,\pi) - P^1 (\Gamma,\pi)$ such that
\[
X = \bigcap_{i=1}^{g-2-r} [V^1 (\Gamma,\pi) + D_i ] \cap V^r (\Gamma,\pi)
\]
is finite, contains $[D]$, and
\[
\# X = 2 \cdot \frac{(g-1)!}{(r+1)!}.
\]

Let $\mathcal{D}_i \in P^r (C,f) - P^1 (C,f)$ be a divisor on $C$ such that $\Trop (\mathcal{D}_i) = D_i$, and let $\mathcal{X} \subseteq \Pic (\WC)$ be
\[
\mathcal{X} = \bigcap_{i=1}^{g-2-r} [V^1 (C,f) + \mathcal{D}_i ] \cap V^r (C,f).
\]
Since $\Trop (\mathcal{X}) \subseteq X$, we see that $\mathcal{X}$ is finite, and $g(C_1) = 1$ by Lemma~\ref{Lem:Nonvacuous}.  By Lemma~\ref{Lem:Intersection}, we have
\[
\# \mathcal{X} \geq 2 \cdot \frac{(g-1)!}{(r+1)!}.
\]

By Proposition~\ref{Prop:LCI}, there exists an open neighborhood $U$ of $[D]$ and effective divisors $E_1 , \ldots , E_{r+1}$ of degree $r-1$ such that
\[
V^r (\Gamma,\pi) \cap U = \Big( \bigcap_{i=1}^{r+1} V^1 (\Gamma,\pi) + E_i - \iota^*E_i \Big) \cap U .
\]
Let $\mathcal{E}_i$ be an effective divisor on $C$ such that $\Trop (\mathcal{E}_i) = E_i$.  By \cite[Proposition~5.4]{PBN}, the Prym-Brill-Noether variety $V^r (C,f)$ is contained in $\cap_{i=1}^{r+1} (V^1 (C,f) + \mathcal{E}_i - \iota^* \mathcal{E}_i)$.

In the open set $U$, the hypersurfaces $\Trop (V^1 (C,f)) + D_i$ and $\Trop (V^1 (C,f) + E_i - \iota^* E_i$ are translates of coordinate tori, and each has multiplicity one by \cite[Lemma~5.1]{PBN}.  By \cite{Rabinoff12}, it follows that there is at most one point tropicalizing to $[D]$ in the intersection of these $g-1$ hypersurfaces.  In other words, the tropicalization map $\Trop \colon \mathcal{X} \to X$ is injective.  But since $\# \mathcal{X} \geq \# X$, it must in fact be bijective.
\end{proof}

We note a few consequences of Theorem~\ref{Thm:Lifting}.

\begin{corollary}
\label{Cor:Reduced}
Let $\Gamma$ be a bielliptic loop of loops of genus $g$ and $\pi \colon \WG \to \Gamma$ the quotient by the antipodal involution.  If $f \colon \WC \to C$ is an \'{e}tale double cover of smooth curves specializing to $\pi$ with $C$ bielliptic, then $V^r (C,f)$ is reduced.  In particular, if $g-1 = r+1$, then $V^r (C,f)$ is smooth.
\end{corollary}

\begin{proof}
The proof of Theorem~\ref{Thm:Lifting} demonstrates that the intersection of $V^r (C,f)$ with $r+1$ general translates of $V^1 (C,f)$ consists of distinct reduced points.  Thus, $V^r (C,f)$ is smooth at these points.  Since $V^r (C,f)$ is Cohen-Macualuay, it has no embedded points, hence it is reduced.
\end{proof}

\begin{corollary}
\label{Cor:Class}
Let $r \geq 2$.  For a general \'{e}tale double cover $f \colon C \to \WC \in \mathcal{R}_{B,g,0}$, one has the following equality of cohomology classes:
\[
[ V^r (C,f) ] =  2 \cdot \frac{\Theta^{r+1}}{(r+1)!} .
\]
\end{corollary}

\begin{proof}
By Lemma~\ref{Lem:Intersection}, there exists an integer $m_r \geq 2$ such that
\[
[ V^r (C,f) ] =  m_r \cdot \frac{\Theta^{r+1}}{(r+1)!} .
\]
Moreover, there exist divisors $D_1 , \ldots , D_{g-r-2} \in P^r (C,f) - P^1 (C,f)$ such that
\[
\mathcal{X} = \bigcap_{i=1}^{g-r-2} [V^1 (C,f) + D_i ] \cap V^r (C,f)
\]
is finite.  The proof of Theorem~\ref{Thm:Lifting} demonstrates that $\mathcal{X}$ consists of exactly $2 \cdot \frac{(g-1)!}{(r+1)!}$ points, and by Corollary~\ref{Cor:Class}, this intersection is reduced.  It follows that $m_r = 2$.
\end{proof}

\bibliography{math}

\end{document}